\documentclass{article}
\usepackage[top=50pt,bottom=60pt,left=70pt,right=70pt]{geometry}
\usepackage{amsmath}
\usepackage{amssymb}
\usepackage{amsthm}
\usepackage{hyperref}
\hypersetup{
  colorlinks   = true, 
  urlcolor     = blue, 
  linkcolor    = blue, 
  citecolor   = red 
}
\usepackage{cleveref}
\usepackage[
backend=biber,
style=alphabetic,
sorting=ynt
]{biblatex}
\newcommand{\Z}{\mathbb{Z}}

\newcommand{\ee}{\boldsymbol{e}}
\newcommand{\ff}{\boldsymbol{f}}
\newcommand{\sbgrp}[1]{\langle #1\rangle}
\newtheorem{Theorem}{Theorem}

\newtheorem{Proposition}[Theorem]{Proposition}
\newtheorem{Lemma}[Theorem]{Lemma}
\theoremstyle{remark}
\newtheorem{Remark}{Remark}
\DeclareMathOperator{\val}{val}

\title{Bounded Generation of Submonoids of Heisenberg Groups}
\author{Doron Shafrir \\
{\small doron.abc@gmail.com}}
\date{}
\begin{document}
\maketitle
\begin{abstract}
   If $G$ is a nilpotent group and $[G,G]$ has Hirsch length $1$, then every f.g.\ submonoid of $G$ is boundedly generated, i.e.\ a product of cyclic submonoids. Using a reduction of Bodart, this implies the decidability of the submonoid membership problem for nilpotent groups $G$ where $[G,G]$ has Hirsch length $2$.
\end{abstract}
\section{Introduction}
\subsection{Overview and Defintions}
All the decision problems we study are uniform: The input includes a finite presentation of a nilpotent group $G$, a description of a certain rational subset $R\subseteq G$, an element $g\in G$, and we have to decide whether $g\in R$. Elements of $G$ can be given as words in the generators. In the case of nilpotent groups, they can be represented more efficiently with a Malcev basis with binary Malcev coordinates.
The relevant decision problems for us are the knapsack problem - given $g,x_1,..,x_n\in G$ is $g\in x_1^*x_2^*\cdots x_n^*$? The submonoid membership problem - given $g\in G$ and a finite $S\subset G$, is $g\in S^*$? And a problem that generalizes both, the membership problem in products of submonoids - given finite subsets $S_1,..,S_n\subset G$ and $g\in G$, does $g\in S_1^*S_2^*...S_n^*$?

Adopting a similar notion from groups, we call a f.g.\ monoid $M$ boundedly generated if $M=x_1^*\cdots x_n^*$ for some finite sequence $x_1,...,x_n\in M$. For a nilpotent group $G$, denote by $h(G)$ the Hirsch length of $G$. We prove that if $G$ is nilpotent and $h([G,G])=1$ (such as the Heisenberg groups $H_{2n+1}(\Z)$), then every f.g.\ submonoid of $G$ is boundedly generated, with a computable generating sequence. This allows us to decide membership in products of submonoids of $G$, which allows us to decide membership in submonoids of nilpotent groups $G$ with $h([G,G])\le2$ by the following reduction:
\begin{Theorem}
\label{thm:dim_gain}
\cite[Theorem A]{bodart2024membership}
Let $G$ be a f.g.\ nilpotent group. The (uniform) submonoid membership problem can be reduced to uniformly deciding membership in products of submonoids in subgroups $H\le G$ that satisfy $[G,G]\subseteq H$  and $h([H,H])<h([G,G])$.
\end{Theorem}
Although the stated reduction in \cite{bodart2024membership} is to the rational subset membership problem, the proof is actually a reduction to the more specific problem of membership in products of submonoids. The reduction is based on the following result: If $M$ is a submonoid of a f.g.\ nilpotent group and $M[G,G]$ is a finite index subgroup of $G$, then $M$ is a finite index subgroup of $G$  \cite{shafrir2024saturation,BodartOrdering}. 

\subsection{Related Results}
\cite[Theorem 6.8]{konig2016knapsack} showed that the knapsack problem is decidable in the groups $H_3(\Z){\times}\Z^e$. We generalize this result to every nilpotent group with $h([G,G])=1$. \cite{colcombet2019reachability} showed that submonoid membership is decidable in Heisenberg groups. In a recent breakthrough, \cite{bodart2024membership} showed that rational subset membership is decidable in $H_3(\Z)$, and conjectured that this holds for all Heisenberg groups. Bodart implicitly proves that f.g.\ submonoids of $H_3(\Z)$ are boundedly generated. On the negative side, submonoid membership is undecidable in some nilpotent group of class $2$ \cite{roman2023undecidability}. Our main result is that f.g.\ submonoids of nilpotent groups with $h([G,G])=1$ are boundedly generated, allowing us to express products of submonoids as knapsacks, and use that to decide membership in this class of rational subsets. This is enough for deciding the submonoid membership problem in nilpotent groups $G$ with $h([G,G])\le 2$  by \Cref{thm:dim_gain}.

\section{Bounded Generation}
If $A$ is a subset of an Abelian group, we define $nA=\underbrace{A+A+\cdots+A}_n$. 

We show that if $A\subset\Z$ is a finite set with bounded gaps, then every sum of elements of $A$, allowing for repetitions, can be expressed as such a sum where most summands are $\min(A)$ and $\max(A)$. Crucially, the maximum number of exceptional summands from the middle of $A$ depends neither on $A$ nor on the number of summands, but only on the maximal gap.
\begin{Lemma}
\label{lem:no_torsion}
    Let $a_1<a_2<...<a_l$ be a sequence in $\Z$ such that $a_{i+1}-a_i\le b$ for $1\le i< l$.
    Set $A=\{a_1,a_2,...,a_l\}$.  Then, for any $n>2b^2$, $nA=(n-2b^2)\{a_1,a_l\}+2b^2A$
\end{Lemma}
\begin{proof}
    Assume $t\in nA$, and let $s\in A^n$ be nondecreasing sequence such that $\sum s_i=t$. We describe an algorithm to transform $s$ to a sequence concetrated at the extremes $\{a_1,a_l\}$ while conserving the sum $t$. We use the notation $[i,j]$ for intervals of integers. While the interval $[i_{-},i_{+}]=\{i\mid a_1<s_i<a_l\}$ has size $>2b^2$, do the following: Define  $g_-:[i_{-},i_{-}+b^2]\rightarrow\{1,..,b\}$ by $g_-(i)=\min\{d\ge1\mid s_i-d\in A\}$, and  $g_+:[i_{+}-b^2,i_{+}]\rightarrow\{1,...,b\}$ by  $g_+(i)=\min\{d\ge1\mid s_i+d\in A\}$.  Since $i_{+}-i_{-}>2b^2$, the domains of $g_-,g_+$ are disjoint. By the pigeonhole principle, there is some $b_-$ with $|g_-^{-1}(b_-)|\ge b$ and $b_+$ with  $|g_+^{-1}(b_+)|\ge b$. Take the smallest $b_+$ indices $i\in g_-^{-1}(b_-)$ and replace $s_i\leftarrow s_i-b_-$, and also take the largest $b_-$ indices $i\in g_+^{-1}(b_+)$ and replace $s_i\leftarrow s_i+b_+$. After the change we still have $s\in A^n$ by definition of $g_\pm$. Since we subtracted $b_+b_-$ and then added $b_-b_+$, $\sum s_i$ is conserved. The resulting sequence is still nondecreasing: since we always bump to the nearest value, monotonicity can only be violated if for some $i$ with $s_i=s_{i+1}$ and bump up $s_i$ but not $s_{i+1}$, or bump down $s_{i+1}$ but not $s_i$. Since we always bump down the smallest relevant indices and bump up the largest indices, this problem will never occur. For example, if $s=(2,2,3,5,5,8,8,8,10,10,10,14...)$, $A=\{2,3,5,8,10...\}$, $b_-=2$, and $b_+=3$, then after the changes we get \space$(2,2,3,\boldsymbol{3},\boldsymbol{3},8,8,8,\boldsymbol{8},10,10,14...)$, reducing the sum by $6$ (this is the lower half of the transformation only). For a fixed coordinate $i$, we observe that $s_i$ either only decrease or only increase during the run, and can therefore change at most $l$ times. Therefore the algorithm halts after $\le nl$ steps, giving a sequence $s\in S^n$ with  $\sum_is_i=t$ and $|\{i\mid a_1<s_i<a_l\}|\le 2b^2$,  demonstrating $t\in (n-2b^2)\{a_1,a_l\}+2b^2A$.
\end{proof}    \begin{Remark}
      Each iteration gives a sequence which is smaller in the lexicographical order and has higher variance. Instead of the constructive proof of \Cref{lem:no_torsion}, we could start with an $s\in A^n$ with $\sum_is_i=t$ which is either lexicographical minimal or has maximal variance and prove that  $|\{i\mid a_1<s_i<a_l\}|\le 2b^2$. 
\end{Remark}
 Although we don't need it for bounded generation, it is interesting to give the dual of \Cref{lem:no_torsion}, where the sum is concentrated at 2 consecutive values instead of the extremes:
\begin{Lemma}
    \label{lem:dual_no_torsion}
    Let $a_1<a_2<...<a_l$ be a sequence in $\Z$ such that $a_{i+1}-a_i\le b$ for $1\le i< l$.
    Set $A=\{a_1,a_2,...,a_l\}$.  Then, for any $n>2b^2$ we have \[nA=\bigcup_{1\le k<l}(n-2b^2)\{a_k,a_{k+1}\}+2b^2A\]
    
\end{Lemma}

\begin{proof}
We sketch the proof. Assume $t\in nA$, and  $s\in A^n$  nondecreasing such that $\sum s_i=t$. As long as $s_{b^2}$ and $s_{n-b^2}$ are neither equal nor consecutive in $A$, do the following: let $k$ be such that $s_{b^2}<a_k<s_{n-b^2}$, and let $[i_-,i_+]=\{i\mid s_i=a_k\}$. Similarly to the above proof, we can bump up some values of $s_1,...,s_{i_{-}-1}$ and bump down some values of $s_{i_{+}+1},..,s_n$ while preserving the sum. Monotonicity is preserved as long as we bump up the largest indices with a given value and bump down the lowest indices with a given value. There is no monotonicity problem on $i_\pm$ since $s$ strictly increase from $i_--1$ to $i_-$ and from $i_+$ to $i_++1$. When the algorithm halts, either $s_{b^2}=s_{n-b^2}=a_k$ or $s_{b^2}=a_k, s_{n-b^2}=a_{k+1}$ for some $k$. Therefore only the first and last $b^2$ elements are outside $\{a_k,a_{k+1}\}$.
\end{proof}

In order to prove bounded generation not just for groups with $[G,G]\simeq\Z$, but also for groups where $[G,G]$ has torsion, we need a slight generalization of the above result.
\begin{Lemma}
\label{lem:torsion}
    Let $G_0$ be a finite Abelian of size $e$, and $A\subset\Z{\times} G_0$ a finite set whose projection on $\Z$ has gaps bounded by $b$. Then, there exists $A_0\subseteq A$ with $|A_0|\le 2e$ such that for any $n>2b^2e$, we have $nA=(n-2b^2e)A_0+2b^2eA$
\end{Lemma}
\begin{proof}
 This is similar to the proof of \Cref{lem:no_torsion}, but now we need more occurrences of each gap to null the torsion part. Let $\pi:\Z{\times} G_0\rightarrow\Z$ be the projection to the first coordinate. Let $B_0=\{\min(\pi(A)),\max(\pi(A))\}$ and $A_0=A\cap\pi^{-1}(B_0)$. We note that $|A_0|\le 2e$ since $|\pi(A_0)|=2$. Let $s\in A^n$ be with $\pi(s)$ nondecreasing.
As long as  $[i_{-},i_{+}]\doteqdot\{i\mid s_i\notin A_0\}$ has more than $2b^2e$ elements, do the following: Define $g_-:[i_{-},i_{-}+b^2e]\rightarrow\{1,...,b\}{\times}G_0$ and $g_+:[i_{+}-b^2e,i_{+}]\rightarrow\{1,...,b\}{\times}G_0$ , such that $s_i\pm g_\pm(i)\in A$ and $\pi(g_\pm(i))$ is minimal. There are $(b_\pm,t_\pm)\in\{1,..,b\}{\times} G_0$ occurring in the image of $g_\pm$ at least $b\cdot e$ times. Take the smallest (resp. largest) $b_\mp\cdot e$ indices $i\in g_\pm^{-1}(b_\pm,t_\pm)$ and replace $s_i\leftarrow s_i\pm (b_\pm,t_\pm)$. One of these operations adds $(b_-b_+e,0)\in\Z{\times}G_0$ to $\sum_i s_i$, and the other subtracts it, so the sum is preserved. When the algorithm halts, we have $|\{i\mid s_i\notin A_0\}|\le 2b^2e$ as needed.
\end{proof}

We now show how all this is relevant to nilpotent groups:
\begin{Proposition}
\label{Prop:reorder}
    Let $G$ be nilpotent of class $2$, $w\in G^l$ a word in $G$ of length $l$, and $x\in G$.  Let $A=\{[x,\val(w_{\le i})]|0\le i\le l\}\subseteq[G,G]$ where $w_{\le i}$ is the $i$'th prefix of $w$.
    For every word $u$ obtained from $w$ by inserting $n$ occurrences of $x$ we have $\val(u)\in \val(x^nw)A^n$.
    Moreover, if  $\val(u)\in\val(x^nw)B^n$ for some subset $B\subset A$, then we can rearrange the occurrences of $x$ in $u$ to $\le|B|$ blocks of $x$ (i.e. powers $x^e$) without changing the value.
\end{Proposition}
The word $x^nw$ has the role of a reference point. The proof is straightforward, noticing that moving $x$ from the beginning of $w$ to the $i$'th position changes the value by $[x,\val(w_{\le i})]$. We can now give our main result:

\begin{Theorem}
\label{thm:bnddgen}
    Let $G$ be nilpotent with $h([G,G])=1$. Then, every f.g.\ submonoid of $G$ is computably boundedly generated.
\end{Theorem}
\begin{proof}
    Identify $[G,G]$ with $\Z{\times} G_0$ where $G_0$ is a finite Abelian group of size $e$. Let $x_1,..,x_n\in G$, we need to show that $M=\{x_1,...,x_n\}^*$ is boundedly generated. Let $b$ be such that $[x_i,x_j]\in\{-b,...,b\}{\times} G_0$ for all $i,j$. We claim that $M=(\bigcup x_i^*)^{4ne(b^2+1)}$. Let $w$ be any word in $\{x_1,..,x_n\}$. We prove by induction on $m\le n$ that we can reorder the occurrences of $x_1,...,x_m$ in $w$ without changing its value, such that there are at most $4me(b^2+1)$ blocks of $x_1,..,x_m$.  In the first step,  let $v_i=\val(w_{\le i})$ and $A=\{[x_1,v_i]\mid 0\le i\le|w|\}$ (one may delete $x_1$ from $w$ before defining $A$; the resulting set is the same). Since $[x_1,v_i]-[x_1,v_{i-i}]=[x_1,w_i]\in \{-b,...,b\}{\times} G_0$, $\pi(A)\subset\Z$ has gaps bounded by $b$. By \Cref{lem:torsion} and \Cref{Prop:reorder}, we can reorder $x_1$ so that $x_1$ appears in at most $2e+2eb^2\le 4e(b^2+1)$ blocks, as needed. In the induction step, we do the same thing with $x_k$, but this time every new block can split an existing block into two, therefore the number of blocks may increase by $4e(b^2+1)$. 
\end{proof}
We wish to stress that the map $i\mapsto [x_k,v_i]$ is not monotonic and may have local minima, therefore bumping down an occurence of $x_k$ to the next smallest value may require moving it far in the word $w$. Similarly, the extreme values $A_0$ of $A$, where the big blocks of $x^k$ are inserted, need not correspond to the endpoints of $w$.

\section{Decidability Results}
It is well-known that every f.g.\ nilpotent group has a torsion-free subgroup of finite index. We sketch how to find such a subgroup for $G$ of class 2.
\begin{Proposition}
\label{prop:torsion_free_fi_subgroup}
    Given a finite presentation of a 2-step nilpotent group $G$, we can find generators of a torsion-free finite index subgroup of $G$.
\end{Proposition}
\begin{proof}
    Take the Abelianization map $\pi:G\rightarrow G/[G,G]\simeq \Z^r{\times} A_0$ with $A_0$ finite, and let $x_1,..,x_r\in G$ be such that $\pi(x_1),..,\pi(x_r)$ form a basis for $\Z^r$. Let $e$ be the exponent of the torsion part of $[G,G]$. Then, $H=\sbgrp{x_1^e,..,x_r^e}$ is finite index since $\pi(H)$ is finite index. Since $\pi(x_1^e),..,\pi(x_r^e)$ generate a free Abelian group, $H\cap[G,G]=[H,H]$. We get that $[H,H]\subseteq [G,G]^{e^2}$ is torsion-free, and $H/[H,H]=H/(H\cap[G,G])\simeq\pi(H)\le\Z^r$ is torsion-free, therefore $H$ is torsion-free. 
\end{proof}
The following is a straightforward generalization of \cite[Theorem 6.8]{konig2016knapsack}
\begin{Theorem}
\label{thm:knapsack}
    The Knapsack Problem is decidable uniformly for all nilpotent groups $G$ with $h([G,G])=1$.
\end{Theorem}
\begin{proof}
    Given a presentation of a nilpotent groups $G$ with $h([G,G])=1$, a sequence $x_1,...,x_n\in G$ and $g\in G$, we need to decide whether $g\in x_1^*\cdots x_n^*$.  By \cite[Theorem 7.3]{konig2016knapsack}, the knapsack problem can be reduced to any subgroup $H$ of $G$ of finite index, and we observe that their reduction can be done uniformly in $G$ and $H$.  Therefore, by \Cref{prop:torsion_free_fi_subgroup} we may assume $G$ is torsion-free.  We build a Malcev basis for $G$: elements $a_1,..,a_r,z\in G$ such that $G/[G,G]=\sbgrp{\pi(a_1)}{\times}\cdots{\times}\sbgrp{\pi(a_r)}$ and $[G,G]=\sbgrp{z}$.
    Define $C:\Z^{r+1}\rightarrow G$ by $C(e_1,...,e_r,f)=a_1^{e_1}\cdots a_r^{e_r}z^f$. Then $C$ is a bijection, giving unique coordinates for $G$. We now express the product in these coordinates: if $C(\ee_3,f_3)=C(\ee_1,f_1)C(\ee_2,f_2)$, then $\ee_3=\ee_1+\ee_2$ and $f_3=f_1+f_2+Q(\ee_1,\ee_2)$ for a quadratic function $Q$ we can compute, as can be seen by reordering $a_1^{e_{1,1}}\cdots a_r^{e_{1,r}}z^{f_1}a_1^{e_{2,1}}\cdots a_r^{e_{2,r}}z^{f_2}$. We have to decide if there is a non-negative solution $\alpha_1,..,\alpha_n$ to $g=x_1^{\alpha_1}\cdots x_n^{\alpha_n}$. Let $x_i=C(\ee_i,f_i)$. Then:
\begin{gather*}
C(\sum\alpha_i\ee_i,\sum\alpha_if_i+f')=C(\ee_1,f_1)^{\alpha_1}C(\ee_2,f_2)^{\alpha_2}\cdots C(\ee_n,f_n)^{\alpha_n}
\\
f'=\sum_i\binom{\alpha_i}{2}Q(\ee_i,\ee_i)+\sum_{i<j}\alpha_i\alpha_jQ(\ee_i,\ee_j)
\end{gather*}
Equating the coordinates of $g$ to those of $x_1^{\alpha_1}\cdots x_n^{\alpha_n}$ we get a system with $r$ linear equations in the variables $\alpha_1,..,\alpha_n$ corresponding to the $a_i$ coordinates, a single quadratic equation in $\alpha_i$ corresponding to the $z$ coordinate, and non-negativity inequalities $\alpha_i\ge 0$. Such as system is decidable by \cite{grunewald2004integer}.
\end{proof}
\begin{Remark}
    Instead of passing to a torsion-free subgroup, we could work with the Malcev coordinates of the given group $G$ directly. In this approach, we need to be careful since coordinates are not unique anymore. Still, one can express $C(\ee_1,\ff_1)=C(\ee_2,\ff_2)$ as a system of equations, resulting in a system of linear equations, a single quadratic equation, and several linear and quadratic modular equations. The congruence classes of solutions to the modular equations can be enumerated, so the system is again decidable by \cite{grunewald2004integer}. 
\end{Remark}

We are ready for the second main result:
\begin{Theorem}
    The submonoid membership problem is decidable for f.g.\ nilpotent groups $G$ with $h([G,G])\le 2$. 
\end{Theorem}
\begin{proof}
    By \cite[Theorem A]{bodart2024membership}, the submonoid membership problem in $G$ can be reduced to finitely many instances of membership in products of submonoids of a subgroup $H\le G$ with $h([H,H])\le 1$. By \Cref{thm:bnddgen}, each submonoid in the product can be replaced by a product of cyclic submonoids, leaving us with an instance of the knapsack problem, which is decidable by \Cref{thm:knapsack}.
\end{proof}

As Bodart points out in \cite{bodart2024membership}, we can also derive decidability of submonoid membership for some nilpotent groups with $h([G,G])=3$, as long as they have no subgroup $H$ with $[G,G]\le H\le G$ such that $h([H,H])=2$.

\section{Open Problems}
Decidability of the rational subset membership problem in Heisenberg groups $H_{2n+1}(\Z)$ for $n\ge 2$ remains open.
Another interesting question is what other groups have the property that every f.g.\ submonoid is boundedly generated, in particular what other nilpotent groups have this property.
Another challenge is to improve the bounds for \Cref{lem:no_torsion} and \Cref{thm:bnddgen}.
The number of cyclic factors in \Cref{thm:bnddgen} is $4ne(b^2+1)$. 
Improving $b$ and $e$ to polylogarithmic factors, would make the number of cyclic submonoids polynomial in the input size, even if the presentation of $G$ and $x_1,...,x_n$  are given efficiently in binary Malcev coordinates. 
If bounded generation is proven for all $2$-step nilpotent groups, the dependence on $b$ and $e$ can be removed entirely, as we may take $x_1,..,x_n$ to be generators of a free nilpotent group.
 
\printbibliography
\end{document}